\documentclass[letterpaper, 10 pt, conference]{ieeeconf}  % Comment this line out
\IEEEoverridecommandlockouts                              % This command is only
\usepackage{amsmath,amssymb} % assumes amsmath package installed
\usepackage{amsthm}
\usepackage{algorithm,algpseudocode}
\usepackage{url,cite}
\usepackage{graphicx,tikz} 
\usepackage[colorlinks=true,linkcolor=blue,citecolor=blue,urlcolor=blue]{hyperref}

\newcommand{\bx}{\boldsymbol{x}}

\newcommand{\bz}{\boldsymbol{z}}

\newcommand{\blambda}{\boldsymbol{\lambda}}

\newcommand{\bQ}{\boldsymbol{Q}}
\newcommand{\bd}{\boldsymbol{d}}
\newcommand{\bff}{\boldsymbol{f}}
\newcommand{\bg}{\boldsymbol{g}}
\newcommand{\bA}{\boldsymbol{A}}

\newcommand{\bH}{\boldsymbol{H}}

\newcommand{\st}{\mathop{\text{\normalfont s.t.}}}
\newcommand{\diag}{\mathop{\text{\normalfont diag}}}

\newcommand{\interior}{\mathop{\text{\normalfont interior}}}

\newtheorem{theorem}{Theorem}
\newtheorem{proposition}{Proposition}
\newtheorem{lemma}{Lemma}
\newtheorem{definition}{Definition}
\newtheorem{assumption}{Assumption}

\newif\iffull
\fulltrue

\title{\LARGE Overlapping Schwarz Decomposition for Constrained Quadratic Programs}

\author{Sungho Shin, Mihai Anitescu, Victor M. Zavala
  \thanks{S. Shin is with the Department of Chemical and Biological Engineering, University of Wisconsin-Madison, Madison, WI 53706 USA (e-mail: sungho.shin@wisc.edu)}
  \thanks{M. Anitescu is with the Mathematics and Computer Science Division, Argonne National Laboratory, Lemont, IL 60439, USA, and also with the Department of Statistics, University of Chicago, Chicago, IL 60637, USA (e-mail: anitescu@mcs.anl.gov)}
  \thanks{V. M. Zavala is with the Department of Chemical and Biological Engineering, University of Wisconsin-Madison, Madison, WI 53706 USA and with the Mathematics and Computer Science Division, Argonne National Laboratory, Lemont, IL 60439, USA (e-mail: victor.zavala@wisc.edu)}}
\begin{document}

\maketitle
\thispagestyle{empty}
\pagestyle{empty}

\begin{abstract}
  We present an overlapping Schwarz decomposition algorithm for constrained quadratic programs (QPs).  Schwarz algorithms have been traditionally used to solve linear algebra systems arising from partial differential equations, but we have recently shown that they are also effective at solving structured optimization problems. In the proposed scheme, we consider QPs whose algebraic structure can be represented by  graphs. The graph domain is partitioned into overlapping subdomains (yielding a set of coupled subproblems), solutions for the subproblems are computed in parallel, and  convergence is enforced by updating primal-dual information in the overlapping regions. We show that convergence is guaranteed if the overlap is sufficiently large and that the convergence rate improves exponentially with the size of the overlap. Convergence results rely on a key property of graph-structured problems that is known as exponential decay of sensitivity. Here, we establish conditions under which this property holds for constrained QPs (as those found in network optimization and optimal control), thus extending existing work that addresses unconstrained QPs. The numerical behavior of the Schwarz  scheme is demonstrated by using a DC optimal power flow problem defined over a network with 9,241 nodes.
\end{abstract}

%%%%%%%%%%%%%%%%%%%%%%%%%%%%%%%%%%%%%%%%%%%%%%%%%%%%%%%%%%%%%%%%%%%%%%%%%%%%%%%%
\section{Introduction}
Structured optimization problems arise in many applications such as {trajectory planning and model predictive control \cite{rawlings2009model}, multistage stochastic programming \cite{yu2019advanced,pereira1991multi}, optimization with embedded partial differential equations (PDEs) \cite{biegler2007real}, and in optimization of energy infrastructures  \cite{lavaei2011zero,coffrin2018powermodels,zlotnik2015optimal}. Decomposition schemes that exploit the problem structure have been used to overcome scalability limits of centralized schemes  \cite{rantzer2009dynamic,conte2012computational,shin2019hierarchical,engelmann2020decomposition,jackson2003temporal}.} A wide range of decomposition schemes have been proposed in the literature such as Lagrangian decomposition \cite{lemarechal2001lagrangian}, the alternating direction method of multipliers (ADMM) \cite{boyd2011distributed}, and Jacobi/Gauss-Seidel methods \cite{shin2018multi}. The basic tenet behind such algorithms is to decompose the original problem into subproblems and to coordinate subproblem solutions by using primal-dual information. {These decomposition schemes are different from waterfilling methods \cite{carli2017distributed} in that the constraints are not necessarily satisfied at  each iteration within the algorithm. Moreover, these schemes are advantageous in that they can handle a wide variety of objective/constraint formulations and in that they have well-understood convergence properties. However, a disadvantage of these schemes is that convergence can be rather slow \cite{kozma2015benchmarking}.}

{\it Overlapping Schwarz} decomposition schemes have been recently used to solve structured optimization problems, and they have  demonstrated to outperform popular schemes such as ADMM and Jacobi/Gauss-Seidel in certain problem classes \cite{shin2020decentralized}. Schwarz algorithms were originally developed for the parallel solution of linear algebra systems arising in PDEs, but such schemes can also be used to handle general linear systems and optimization problems by exploiting their underlying algebraic topology  \cite{frommer2001algebraic}. As the name suggests, overlapping Schwarz algorithms decompose the full problem (the underlying graph) into {\it subproblems} that are defined over overlapping subdomains. Solutions for the subproblems are computed in parallel and convergence is enforced by updating primal-dual information in the overlapping regions.  In the context of quadratic programs (QP) that are unconstrained and convex, we have shown that the convergence rate of Schwarz algorithms improves exponentially with the size of the overlapping region \cite{shin2020decentralized}. Overlapping Schwarz schemes provide a bridge between fully decentralized Jacobi/Gauss-Seidel algorithms (no overlap) and centralized algorithms (the overlap is the entire domain). 

This paper presents an overlapping Schwarz algorithm for constrained QPs. We analyze the convergence properties of the algorithm and derive an explicit relationship between its convergence rate and the size of the overlap. {This result extends existing convergence results for unconstrained QPs \cite{shin2020decentralized} to equality- and inequality-constrained QPs}. In particular, we show that the algorithm converges with sufficiently large overlap and that the convergence rate improves exponentially with the size of overlap. {This convergence result relies on a property called {\it exponential decay of sensitivity} \cite{na2019exponential,shin2020diffusing,grüne2019sensitivity,grune2020exponential,shin2020decentralized}. The property states that the sensitivity of the primal-dual solution at a given node decays exponentially with respect to the distance from the perturbation. Such a property has been established for discrete-time \cite{na2019exponential,shin2020diffusing} and continuous-time \cite{grüne2019sensitivity,grune2020exponential} optimal control problems (the graph is a line) and for unconstrained QPs (general graph) \cite{shin2020decentralized}. This paper establishes the exponential sensitivity property for constrained QPs over general graphs, thus making the results broadly applicable.}

The paper is organized as follows. In the remainder of this section we introduce basic notation and the problem under study. In Section \ref{sec:alg} we introduce the overlapping Schwarz algorithm. In Section \ref{sec:sens} we present the main theoretical results. We first analyze the sensitivity of the solution of structured QPs against parametric perturbations and then use the results to establish  convergence conditions for the algorithm. Numerical results are given in Section \ref{sec:num}. \iffull\else {Some of the technical lemmas and proofs are omitted due to page limitations; these results can be found in the full version of the paper, available at \url{https://arxiv.org/abs/2003.07502}}.\fi

{\it Notation}. The set of real numbers and the set of integers are denoted by $\mathbb{R}$ and $\mathbb{I}$, respectively, and we define $\mathbb{I}_{a:b}:=\mathbb{I}\cap[a,b]$, $\mathbb{I}_{>0}:=\mathbb{I}\cap(0,\infty)$, $\mathbb{R}_{>0}:=(0,\infty)$, and $\overline{\mathbb{R}}:=\mathbb{R}\cup\{\infty\}$. By default, vectors are assumed to be column vectors, and we use the syntax 
$(M_1,\cdots,M_n):=[M_1^\top\,\cdots\,M_n^\top]^\top$, 
$\{M_i\}_{i\in \mathcal{I}}:=(M_{i_1}, \cdots, M_{i_{m}})$, and
$\{M_{i,j}\}_{i\in \mathcal{I},j\in \mathcal{J}}:=\{\{M_{i,j}^\top\}_{j\in\mathcal{J}}^\top\}_{i\in \mathcal{I}}$, where $\mathcal{I}=\{i_1<\cdots<i_{m}\}$ and $\mathcal{J}=\{j_1<\cdots<j_{n}\}$.
Furthermore, $v[i]$ is the $i$th component of $v$, $M[i,j]$ is the $(i,j)$th component of $M$,
$v[\mathcal{I}]:=\{v[i]\}_{i\in \mathcal{I}}$, and
$M[\mathcal{I},\mathcal{J}]:=\{M[i,j]\}_{i\in \mathcal{I},j\in \mathcal{J}}$.
Vector 2-norms and induced 2-norms are denoted by $\Vert\cdot\Vert$.
For matrices $A$ and $B$, $A\succeq B$ indicates that $A-B$ is positive semi-definite while $A\geq B$ represents a componentwise inequality.

{\it Setting}. We consider a (potentially infinite) {\it parent graph} $\mathcal{G}(\mathcal{V},\mathcal{E})$, where $\mathcal{V}$ is the set of nodes and $\mathcal{E}$ is the set of edges. We also consider the finite node subset $U\subseteq \mathcal{V}$ and the following QP:
\begin{subequations}\label{eqn:qp}
  \begin{align}
    \;\min_{\{x_i\}_{i\in U}}\; &\sum_{i\in U}\sum_{j\in N_U[i]}\frac{1}{2}x_i^\top Q_{i,j}x_j - \sum_{i\in U}f_i^\top x_i\label{eqn:qp-obj}\\
    \st\;& \sum_{j\in N_U[i]} A^{\mathcal{E}}_{i,j}x_j = g^{\mathcal{E}}_i,\; (\lambda^{\mathcal{E}}_i),\; i\in U\label{eqn:qp-con-e}\\
    & \sum_{j\in N_U[i]} A^{\mathcal{I}}_{i,j}x_j \geq g^{\mathcal{I}}_i,\; (\lambda^{\mathcal{I}}_i),\; i\in U.\label{eqn:qp-con-i}
  \end{align}
\end{subequations}
Here $x_i\in\mathbb{R}^{r_i}$ are the decision variables for node $i$; $\lambda^{\mathcal{E}}_i\in\mathbb{R}^{m^{\mathcal{E}}_i}$ and $\lambda^{\mathcal{I}}_i\in\mathbb{R}^{m^{\mathcal{I}}_i}$ are the dual variables; $Q_{i,j}\in\mathbb{R}^{r_i\times r_j}$, $A^{\mathcal{E}}_{i,j}\in\mathbb{R}^{m^{\mathcal{E}}_i \times r_j}$, $A^{\mathcal{I}}_{i,j}\in\mathbb{R}^{m^{\mathcal{I}}_i \times r_j}$, $f_i\in\mathbb{R}^{r_i}$, $g^{\mathcal{E}}_i\in\mathbb{R}^{m^{\mathcal{E}}_i}$, and $g^{\mathcal{I}}_i\in\mathbb{R}^{m^{\mathcal{I}}_i}$ are the data; and $N_U[X]:=N_U(X)\cup X$, where $N_U(X):=\{j\in U\setminus X: \exists i\in X\text{ such that }\{i,j\}\in\mathcal{E}\}$ and the argument is considered as a singleton if $X$ is a single node. We define $A_{i,j}:=(A^{\mathcal{E}}_{i,j},A^{\mathcal{I}}_{i,j})$, $g_i:=(g^{\mathcal{E}}_i,g^{\mathcal{I}}_i)$, $\lambda_i:=(\lambda^{\mathcal{E}}_i,\lambda^{\mathcal{I}}_i)$, $z_i:=(x_i,\lambda_i)$, $d_i:=(f_i,g_i)$, $m_i:=m^{\mathcal{E}}_i + m^{\mathcal{I}}_i$, and $n_i:=r_i+m_i$. We assume that $Q_{i,j}=Q_{j,i}^\top$.

An equivalent problem can be written in a compact form:
\begin{align*}
  P_{U}(\bd_{U}):\; \min_{\bx_{U}}\; &\frac{1}{2}\bx_U^\top\bQ_U \bx_U - \bff^\top_U \bx_U\\
  \st\;& \bA^{\mathcal{E}}_{U} \bx_U = \bg^{\mathcal{E}}_U, \; (\blambda^{\mathcal{E}}_U)\\
  & \bA^{\mathcal{I}}_{U} \bx_U \geq \bg^{\mathcal{I}}_U, \; (\blambda^{\mathcal{I}}_U).
\end{align*}
Here, $\bx_U:=\{x_i\}_{i\in U}$; $\blambda^{\mathcal{E}}_U:=\{\lambda^{\mathcal{E}}_i\}_{i\in U}$; $\blambda^{\mathcal{I}}_U:=\{\lambda^{\mathcal{I}}_i\}_{i\in U}$; $\blambda_U:=\{\lambda_i\}_{i\in U}$; $\bz_U:=\{z_i\}_{i\in U}$; $\bff_U:=\{f_i\}_{i\in U}$; $\bg^{\mathcal{E}}_U:=\{g^{\mathcal{E}}_i\}_{i\in U}$; $\bg^{\mathcal{I}}_U:=\{g^{\mathcal{I}}_i\}_{i\in U}$; $\bg_U:=\{g_i\}_{i\in U}$; $\bd_U:=\{d_i\}_{i\in U}$; $\bQ_{U}=\{Q_{i,j}\}_{i,j\in U}$, $\bA^{\mathcal{E}}_{U}:=\{A^{\mathcal{E}}_{i,j}\}_{i,j\in U}$; $\bA^{\mathcal{I}}_{U}:=\{A^{\mathcal{I}}_{i,j}\}_{i,j\in U}$; $\bA_{U}:=\{A_{i,j}\}_{i,j\in U}$; $r_U:=\sum_{i\in U}r_i$; $m_U=\sum_{i\in U}m_i$; and $n_U=\sum_{i\in U} n_i$. The problem is denoted as the parametric form $P_U(\bd_U)$. 

\section{Overlapping Schwarz Algorithm}\label{sec:alg}
This section introduces the Schwarz algorithm for the solution of $P_V(\bd_V)$ (referred to as the {\it full problem}) with $V\subseteq \mathcal{V}$. We consider a {\it non-overlapping} partition $\{V_k\subseteq V\}_{k=1}^K$ of $V$ and an {\it overlapping partition} $\{W_k\subseteq V\}_{k=1}^K$ of $V$ such that $V_k\subseteq W_k$ holds for $k\in\mathbb{I}_{1:K}$. We call $V_1,\cdots,V_K$ {\it original subdomains} and $W_1,\cdots,W_K$ {\it expanded subdomains}. The Schwarz algorithm is defined below.

\begin{algorithm}\caption{Overlapping Schwarz Algorithm}\label{alg:main}
  \begin{algorithmic}[1]
    \Require $\bz^{(0)}_V$, $\{V_k\}_{k=1}^K$, $\{W_k\}_{k=1}^K$
    %% , $\bQ$, $\bA$, $\bd$, $\epsilon_{\text{tol}}$, $\ell_{\max}$
    %% \For {(in parallel) $k\in\mathbb{I}_{1:K}$}
    %% \State Build $P_{W_k}(\bd)\leftarrow \bQ_{W_k,W_k},\bA_{W_k,W_k},\bd_{W_k}$
    %% \EndFor
    \For {$\ell=0,1,\cdots$}
    \For {(in parallel) $k=1$ to $K$}
    %% \State Update RHS \& cost: $P_{W_k}(\bd_{W_k}-\bH_{-W_k}\bz^{(\ell)}_{-W_k})$
    %% \State Solve $P_{W_k}$ to obtain $\bz^*_{W_k}(\bd_{W_k}-\bH_{-W_k}\bz^{(\ell)}_{-W_k})$
    \State $\bz^{(\ell+1)}_{V_k}=T_{V_k\leftarrow W_k}\bz^*_{W_k}(\bd_{W_k}-\bH_{-W_k}\bz^{(\ell)}_{-W_k})$
    \EndFor
    %% \State $\epsilon\leftarrow$ Residual to \eqref{eqn:kkt-cen}, $\ell\leftarrow \ell+1$
    \EndFor
    \Ensure $\bz^{(\ell)}_V$
  \end{algorithmic} 
\end{algorithm}
Here, we use a syntax that can be applied to any $U\subseteq \mathcal{V}$: $\bH_{U}:=\{H_{i,j}\}_{i,j\in U}$; $\bH_{-U}:=\{H_{i,j}\}_{i\in U,j\in N_V(U)}$; $H_{i,j}:=[Q_{i,j}\;A_{j,i}^\top; A_{i,j}\;0]$; and $\bz_{-U}:=\{z_{i}\}_{i\in N_V(U)}$. Furthermore, $\bz^*_U(\cdot)$ is the  primal-dual solution mapping of the parametric optimization problem $P_{U}(\cdot)$; $T_{U_1\leftarrow U_2}:=\{T_{i,j}\}_{i\in U_1,j\in U_2}$, where $U_1,U_2\subseteq \mathcal{V}$ and $T_{i,j}=I_{n_i\times n_i}$ if $i=j$ and $0_{n_i\times n_j}$ otherwise. Note that $\bz_{-U}$ is supposed to represent the solution information that is {\it complementary} to $U$. The full complementary solution information includes the solution on $V\setminus U$. However, the variables and constraints in $U$ are coupled only with $N_V(U)$, so it suffices to incorporate information only for the {\it coupled complementary} region $N_V(U)$. Therefore, we will abuse the term complementary solution to represent the coupled complementary solution $\bz_{-U}$.

The core part of the algorithm (line 3) consists of three steps:  subproblem solution, solution restriction, and  primal-dual exchange. In the first step, one formulates the {\it subproblem} for the $k$th subdomain as $P_{W_k}(\bd_{W_k}-\bH_{-W_k} \bz^{(\ell)}_{-W_k})$ (this formulation will be justified later in Lemma \ref{lem:con}). The subproblem incorporates complementary solution information $\bH_{-W_k} \bz^{(\ell)}_{-W_k}$, which is obtained during the third inner step of the previous iteration step. The subproblem is solved to obtain its solution $\bz^*_{W_k}(\bd_{W_k}-\bH_{-W_k} \bz^{(\ell)}_{-W_k})$. Here, we observe that solution multiplicity exists at the overlapping region. To remove such multiplicity, we {\it restrict} the solution in the second step. In particular, we abandon the primal-dual solutions associated with $W_k\setminus V_k$ (subdomain region acquired by expansion) and take only those solutions associated with $V_k$ (the original subdomain region). This procedure is represented by the restriction operator $T_{V_k\leftarrow W_k}$. After restriction, the solutions are assembled over $k\in\mathbb{I}_{1:K}$ to make the next guess of the solution $\bz^{(\ell+1)}_V$. In the third step, the primal-dual solutions  are exchanged across the subdomains to update the complementary information $\bH_{-W_k} \bz^{(\ell)}_{-W_k}$ for each subproblem. The schematic of the algorithm is shown in Fig. \ref{fig:schematic}. The algorithm can be implemented in a fully decentralized manner, and different updating schemes can be used (e.g., Gauss-Seidel or asynchronous) \cite{shin2020decentralized}.

\begin{figure}[t]
  \centering
  \vspace{0.07in}
  \begin{tikzpicture}
    \node[circle,fill,scale=.4] (A1) at (-.5,0) {};
    \node[circle,fill,scale=.4] (A2) at (-1,2/3) {};
    \node[circle,fill,scale=.4] (A3) at (-.5,4/3) {};
    \node[circle,fill,scale=.4] (A4) at (-2.2,2/3) {};
    \node[circle,fill,scale=.4] (A5) at (-2.7,0) {};
    \node[circle,fill,scale=.4] (A6) at (-2.7,4/3) {};

    \node[circle,fill,scale=.4] (B1) at (.7,-4/3) {};
    \node[circle,fill,scale=.4] (B2) at (1.2,-2/3) {};
    \node[circle,fill,scale=.4] (B3) at (.7,0) {};
    \node[circle,fill,scale=.4] (B4) at (2.4,-2/3) {};
    \node[circle,fill,scale=.4] (B5) at (2.9,-4/3) {};
    \node[circle,fill,scale=.4] (B6) at (2.9,0) {};

    \draw (A1)--(A2) (A2)--(A3) (A2)--(A4) (A4)--(A5) (A4)--(A6);
    \draw (A1)--(B3);
    \draw (B1)--(B2) (B2)--(B3) (B2)--(B4) (B4)--(B5) (B4)--(B6);

    \draw[red,transform canvas={xshift=.5em,yshift=.2em},->] (A2) -- (A1);
    \node at (-.3,2/3) {$\color{red} \bz^{(\ell)}_{-2}$};
    \draw[blue,transform canvas={xshift=-.5em,yshift=-.2em},->] (B2) -- (B3);
    \node at (.5,-2/3) {$\color{blue} \bz^{(\ell)}_{-1}$};

    \draw[rounded corners=5pt,red,opacity=0.2,line width=4pt](-3.2,-.3) rectangle ++(4.4,2.6);
    \draw[rounded corners=5pt,blue,opacity=0.2,line width=4pt](-1,-2.3) rectangle ++(4.4,2.6);
    \draw[rounded corners=5pt,red,opacity=0.5,line width=2pt](-3.15,-.2) rectangle ++(3.2,2.4);
    \draw[rounded corners=5pt,blue,opacity=0.5,line width=2pt](.15,-2.2) rectangle ++(3.2,2.4);
    \node at (-.5,-1.8) {$W_2$};
    \node at (.7,1.8) {$W_1$};
    \node at (.7,-1.8) {$V_2$};
    \node at (-.5,1.8) {$V_1$};
  \end{tikzpicture}
  \caption{Schematic representation of the overlapping Schwarz algorithm.}\label{fig:schematic}
\end{figure}
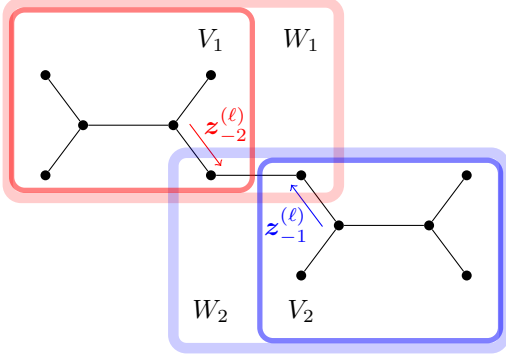

\section{Main Results}\label{sec:sens}
In this section we analyze the convergence of the Schwarz algorithm. We will see that parametric sensitivity plays a central role in convergence behavior because, intuitively, primal-dual solutions of the neighbors of a subdomain enter as parametric perturbations. We first analyze the parametric sensitivity of $P_U(\cdot)$ in a way that the result can be generally applied to any node subset $U\subseteq \mathcal{V}$. We then exploit this sensitivity result to establish convergence.

\subsection{Preliminaries}
\begin{assumption}\label{ass:sol}
  The following is assumed for $U\subseteq \mathcal{V}$.
  \begin{itemize}
  \item[(a)] $\bQ_U\succeq 0$.
  \item[(b)] Given a convex set $\mathbb{D}_U\subseteq \mathbb{R}^{m_U}$, for any $P_U(\bd_U)$ with $\bd_U\in \mathbb{D}_U$, there exists a primal-dual solution at which the second-order sufficient condition (SOSC) and linear independence constraint qualification (LICQ) hold.
  \end{itemize}  
\end{assumption}

By Assumption \ref{ass:sol}, there exists a unique primal-dual solution of $P_U(\bd_U)$ for any $\bd_U\in\mathbb{D}_U$. Thus, the primal-dual solution mapping $\bz^*_U: \mathbb{D}_U\rightarrow \mathbb{R}^{n_U}$ is well defined.
%% In our setting, we ensure that Assumption \ref{ass:sol} holds by using convexity of the QP and by assuming that the solution satisfies the linear independence constraint qualification (LICQ). Assumption \ref{ass:sol} guarantees that the solution mapping $\bz^*_U: \mathbb{D}_U\rightarrow \mathbb{R}^{n}$ is well-defined and is standard in perturbation results for optimization problems.  

\begin{definition}\label{def:basic}
  Consider $P_U(\cdot)$ and $B\subseteq \mathbb{I}_{1:n_U}$.
  \begin{itemize}
  \item[(a)] $B$ is called a basis if $\bH_{U}[B,B]$ is nonsingular.
  \item[(b)] $\bz_U(\bd_U)\in\mathbb{R}^{n_U}$ is called a basic solution of $P_U(\bd_U)$ if
    \begin{subequations}\label{eqn:basic-sol}
      \begin{align}
        \bH_U[B,B]\bz_U(\bd_U)[B] &=\bd_U[B],\\
        \bz_U(\bd_U)[\mathbb{I}_{1:n_U}\setminus B] &= 0.
      \end{align}
    \end{subequations}
    \item[(c)] $B$ is feasible (optimal) for $P_U(\bd_U)$ if its basic solution is feasible (optimal) for $P_U(\bd_U)$.
  %% \item[(c)] $\mathbb{B}_{V}(D_V)$ is the set of bases whose basic solution is optimal to $P_V(\bd_V)$ for some $\bd_V\in D_V$.
\end{itemize}
\end{definition}
Definition \ref{def:basic} is an extension of the notion of basis for linear programs (studied in \cite{shin2020diffusing}). Note that the basis associated with a basic solution may not be unique for the case where there exist zero components in $\bx^*_U(\bd_U)$ or strict complementarity does not hold. We consider $\bz^B_U:\mathbb{R}^{m_U}\rightarrow \mathbb{R}^{n_U}$ as the basic solution mapping for $B$.

\iffull\begin{lemma}\label{lem:basic-1}
  Let Assumption \ref{ass:sol} hold. For $\bd_U\in\mathbb{D}_U$, there exists $B\in\mathbb{B}_{U}$ that is a basis for $\bz^*_U(\bd_U)$, where $\mathbb{B}_U:=\{B\subseteq\mathbb{I}_{1:n_U}:\exists \bd_U\in\mathbb{D}_{U}\text{ such that } B \text{ is optimal for }P_U(\bd_U)\}$.
\end{lemma}
\iffull\begin{proof}
  We let
  \begin{align}\label{eqn:mat}
    \widehat{\bH}:=
    \begin{bmatrix}
      \bQ_{U}& \bA_{U}[\mathcal{A}^*(\bd_U),:]^\top\\
      \bA_{U}[\mathcal{A}^*(\bd_U),:]
    \end{bmatrix},
  \end{align}
  where $\mathcal{A}^*(\bd_U)$ is the active constraint set evaluated at the solution of $P_U(\bd_U)$. 
  From Assumption \ref{ass:sol}, $\bA_{U}[\mathcal{A}^*(\bd_U),:]$ has full row-rank by LICQ, and the reduced Hessian associated with $\widehat{\bH}$ is positive definite by SOSC. These imply that $\widehat{\bH}$ is nonsingular \cite[Lemma 16.1]{nocedal2006numerical}. We choose $B:=\{i\in\mathbb{I}_{1:n_U}: \bz^*_U(\bd_U)[i]\neq 0\}$. One can show that $\bH_U[B,B]$ is a permuted principal submatrix of $\widehat\bH$. Therefore, ${\bH}_U[B,B]$ is nonsingular, and this yields that $B$ is a basis. From the KKT conditions for $P_U(\bd_U)$ and the definition of $B$, \eqref{eqn:basic-sol} holds for $\bz^*_U(\bd_U)$. Thus, we have $\bz^B_U(\bd_U)=\bz^*_U(\bd_U)$.
\end{proof}\fi

\begin{lemma}\label{lem:quad}
  Let $q_1,\cdots,q_{N_q}:(a,b)\rightarrow \mathbb{R}$ be quadratic functions. We let $q:(a,b)\rightarrow \mathbb{R}$ be $q(\cdot):=\min_{i\in\mathbb{I}_{1:{N_q}}} q_i(\cdot)$. Then there exists $\{a_0=a,a_1,\cdots,a_{K_q}=b\}$ such that for each $k\in\mathbb{I}_{1:{K_q}}$, there exists $i\in\mathbb{I}_{1:{N_q}}$ such that $q_i(\cdot)=q(\cdot)$ on $(a_{k-1},a_{k})$.
\end{lemma}
\begin{proof}
  For each $(i,j)\in \mathbb{I}_{1:{N_q}}\times \mathbb{I}_{1:{N_q}}$, we let $I_{i,j}:=\{x\in(a,b): q_i(x)\leq q_j(x)\}$. Since $q_i(\cdot)$ and $q_j(\cdot)$ are quadratic, we have that $I_{i,j}$ is obtained as a union of intervals (not necessarily open or closed). Then we can define $I_i:=\bigcap_{j\in\mathbb{I}_{1:{N_q}}} I_{i,j}$. Since $I_{i,j}$ are unions of intervals, we have that $I_i$ is also a union of intervals. By collecting the end points of the intervals in $\{I_{i}:i\in\mathbb{I}_{1:{N_q}}\}$, we can construct $\{a_0,\cdots,a_{K_q}\}$. In particular, $\{a_0,\cdots,a_{K_q}\}=\bigcup_{i\in\mathbb{I}_{1:{N_q}}}\text{closure}(I_i)\setminus \interior(I_i)$. We observe that $\bigcup_{i\in\mathbb{I}_{1:{N_q}}} I_i= (a,b)$; thus,  for any $k\in\mathbb{I}_{1:{K_q}}$, $(a_{k-1},a_k)\subseteq I_i$ for some $i\in\mathbb{I}_{1:{N_q}}$. This means $q_i(\cdot)=q(\cdot)$ on $(a_{k-1},a_k)$. The proof is complete.
 \end{proof}\fi

\begin{lemma} \label{lem:basic-2}
  Let Assumption \ref{ass:sol} hold. For $\bd_U,\bd_U'\in \mathbb{D}_U$, there exist $\{s_0=0<\cdots<s_{N_d}=1\}$ and $\{B_k\in\mathbb{B}_{U}\}_{k=1}^{N_d}$ such that for $k\in\mathbb{I}_{1:N_d}$, $B_k$ is optimal for $P_U(\bd^s_U)$ with any $s\in[s_{k-1},s_{k}]$, where $\bd^s_U:=(1-s)\bd_U+s\bd'_U$.
\end{lemma}
\iffull\begin{proof}
  Let $\mathbb{B}_U=\{B_{(1)},\cdots,B_{(T)}\}$ (note that $\mathbb{B}_U$ is finite). We define $\pi_{(t)}:[0,1]\rightarrow\overline{\mathbb{R}}$ to be the mapping from $s\in[0,1]$ to the objective value of $\bz^{B_{(t)}}_U(\bd^s_U)$ for $P_U(\bd^s_U)$ (the value is $+\infty$ if $B_{(t)}$ is infeasible). Also, we define $\pi:[0,1]\rightarrow\overline{\mathbb{R}}$ by $\pi(s):=\min_{t\in\mathbb{I}_{1:T}}\pi_{(t)}(s)$. By Lemma \ref{lem:basic-1}, $\pi(\cdot)$ is the objective value mapping of $P_U(\bd^s_U)$ from $s\in[0,1]$.
  
  %% Now we define $\Pi_{(t)}:=\{s\in[0,1]:\pi_{(t)}(s)=\pi(s)\}$. Here we claim that $\Pi_{(t)}$ is obtained as a finite union of closed intervals in $[0,1]$. The proof of the claim is as follows.
  One can see that $z^{B_{(t)}}_U(\bd^s_U)$ is affine in $s$; thus the feasibility conditions for $z^{B_{(t)}}_U(\bd^s_U)$ can be expressed by a finite number of affine equalities and inequalities. This implies that the set of $s\in[0,1]$ on which $\pi(s)<\infty$ is obtained as a closed interval in $[0,1]$. Accordingly, $\pi_{(t)}(\cdot)$ is a quadratic function on a closed interval support. Now, we collect the endpoints of such intervals over $t\in\mathbb{I}_{1:T}$ to construct $\widetilde{\Pi}:=\{\widetilde{s}_0=0,\cdots,\widetilde{s}_{\widetilde{N}_d}=1\}$. For each $k\in\mathbb{I}_{1:\widetilde{N}_d}$, we collect $\mathcal{T}_k:=\{t\in\mathbb{I}_{1:T}:\pi_{(t)}(\cdot)<\infty$ on $(\widetilde{s}_{k-1},\widetilde{s}_{k})\}$. Observe that (i) $\pi_{(t)}$ with $t\in\mathcal{T}_k$ are quadratic on $(\widetilde{s}_{k-1},\widetilde{s}_{k})$ and (ii) $\pi(\cdot)=\min_{t\in\mathcal{T}_k}\pi_{(t)}(\cdot)$ on $(\widetilde{s}_{k-1},\widetilde{s}_{k})$. By Lemma \ref{lem:quad} (stated below, applicable due to (i)), we have that each $(\widetilde{s}_{k-1},\widetilde{s}_{k})$ can be further divided by using $\widetilde{\widetilde{\Pi}}_{k}:=\{\widetilde{\widetilde{s}}_{k,0}=\widetilde{s}_{k-1},\cdots,\widetilde{\widetilde{s}}_{k,\widetilde{\widetilde{N}}_{d,k}}=\widetilde{s}_k\}$, where for each $(\widetilde{\widetilde{s}}_{k,k'-1},\widetilde{\widetilde{s}}_{k,k'})$ with $k'\in\mathbb{I}_{1:\widetilde{\widetilde{N}}_{d,k}}$ and $k\in\mathbb{I}_{1:\widetilde{N}_d}$, there exists $t\in\mathcal{T}_k$ such that $\pi_{(t)}=\min_{t\in\mathcal{T}_k}\pi_{(t)}(\cdot)=\pi(\cdot)$ (recall observation (ii)). 

We now let $\{s_0,\cdots,s_{N_d}\}=\bigcup_{k\in\mathbb{I}_{1:\widetilde{N}_d}}\widetilde{\widetilde{\Pi}}_{k}$. One can observe that, for each $k\in\mathbb{I}_{1:N_d}$, there exists $t\in\mathbb{I}_{1:T}$ such that $\pi_{(t)}(\cdot)=\pi(\cdot)$ on $(s_{k-1},s_k)$. We choose such $B_{(t)}$ as $B_k$; it is known that the objective value mapping of a QP is continuous on its support (e.g., see \cite[Corollary 9]{hadigheh2007sensitivity} or \cite[Theorem 5.53]{bonnans2013perturbation}); thus $\pi(\cdot)$ is continuous on its support. By the continuity of $\pi(\cdot)$ and $\pi_{(t)}(\cdot)$ on their supports, we have that $\pi_{(t)}(\cdot)=\pi(\cdot)$ on $(s_{k-1},s_k)$ implies that the same holds on $[s_{k-1},s_k]$. Finally, one can check that $B_k$ is optimal for $P_U(\bd^s_U)$ with $s\in[s_{k-1},s_k]$. The desired $\{s_k\}_{k=0}^{N_d}$ and $\{B_k\}_{k=1}^{N_d}$ are thus obtained.
  
  %% By collecting such closed intervals, one can construct $\Pi_{(t)}$ as a finite union of closed intervals. The proof of the claim is complete.
  
  %% Now let $\Pi:=\bigcup_{t\in\mathbb{I}_{1:T}} \Pi_{(t)}\setminus \interior(\Pi_{(t)})$. One can see that $0,1\in\Pi\subseteq [0,1]$ and $\Pi$ is finite. We choose $\{s_0,\cdots,s_{N_d}\}=\Pi$. 
\end{proof}\fi

An important implication of Lemma \ref{lem:basic-2} is that the solution path obtained by the perturbation on $\bd_U$ can be divided into multiple paths, each of which is a basic solution mapping. Thus, given Lemma \ref{lem:basic-2}, it suffices to  study the sensitivities of the basic solution mappings.

\subsection{Exponential Decay of Sensitivity}
We now establish our main sensitivity result for the constrained QP, known as exponential decay of sensitivity.
\begin{theorem}[Exponential Decay of Sensitivity]\label{thm:eds}
  Let Assumption \ref{ass:sol} hold. The following holds for $\bd_U,\bd'_U\in \mathbb{D}_U$:
  \begin{align*}
    \left\Vert T_{i\leftarrow U}\left(\bz^*_{U}( \bd)-\bz^*_{U}( \bd')\right)\right\Vert &= \sum_{j\in U}\Gamma_U \rho_U^{\lceil \frac{\Delta_U(i,j)-1}{2}\rceil}\Vert d_{j}-d'_{j}\Vert 
  \end{align*}
with $\Gamma_U:=\overline{\sigma}_U/\underline{\sigma}^2_U$;
  $\rho_U := (\overline{\sigma}_U^2-\underline{\sigma}_U^2)/(\overline{\sigma}_U^2+\underline{\sigma}_U^2)$;
  $\displaystyle\underline{\sigma}_U:=\min_{B\in\mathbb{B}_U}  \sigma_{\min}(\bH_U[B,B])$;
  $\displaystyle\overline{\sigma}_U:=\max_{B\in\mathbb{B}_U} \sigma_{\max}(\bH_U[B,B])$.
\end{theorem}
Here $\lceil\cdot\rceil$ denotes the ceiling operator, and $\Delta_U(i,j)$ denotes the geodesic distance between $i,j\in U$ on the subgraph of $\mathcal{G}(\mathcal{V},\mathcal{E})$ induced by $U$ (the number of elements in the shortest path $\{e_q\in\mathcal{E}: e_q\subseteq  U\}_{q=1}^{n_e}$ from $i$ to $j$); $\sigma_{\min}$ and $\sigma_{\max}$ denote the minimum and maximum singular values of their matrix argument.

 To facilitate the discussion, we define for $M\in\mathbb{R}^{|B|\times |B|}$ and $v\in\mathbb{R}^{|B|}$ the following: $M_{[i][j]}:=M[\mathcal{I}_i,\mathcal{I}_j]$, $v_{[i]}:=v[\mathcal{I}_i]$, where $\mathcal{I}_i:=\{q\in\mathbb{I}_{1:|B|}:b_q\in\{\sum_{k\in U,k<i}n_k+1,\cdots,\sum_{k\in U,k\leq i}n_k\}\}$, and $B=\{b_1<\cdots<b_{|B|}\}$. Note that $\mathcal{I}_{i}$ is the index set that extracts the indices associated with $z_i$, and $\{\mathcal{I}_{i}\}_{i\in U}$ partitions $\mathbb{I}_{1:|B|}$. 

\iffull\begin{lemma}\label{lem:claim}
  If $\Delta_U(i,j)>q\in\mathbb{I}_{>0}$, $({\bH}_U[B,B]^q)_{[i][j]} = 0$.
\end{lemma}
\begin{proof}
  We use the notation $\widetilde{\bH}:=\bH_U[B,B]$. We proceed by induction. From the fact that $H_{i,j}=0$ if $\Delta_U(i,j)>1$, one can observe that $\widetilde{\bH}_{[i][j]}=0$ if $\Delta_U(i,j)>1$. Hence, the claim holds for $q=1$. Now suppose the claim holds for $q=q'$. One can easily see that triangle inequality holds for distance $\Delta_U(\cdot,\cdot)$. From triangle inequality, if $\Delta_U(i,j)>q'+1$, either $\Delta_U(i,l)>q'$ or $\Delta_U(j,l)>1$ holds for any $l\in U$. Thus, if $\Delta_U(i,j)>q'+1$, then $(\widetilde{\bH}^{q'+1})_{[i][j]} = \sum_{l\in U} (\widetilde{\bH}^{q'})_{[i][l]}\widetilde{\bH}_{[l][j]} = 0$. 
\end{proof}\fi

\begin{lemma}\label{lem:inv}
  \small $\Vert( \bH_U[B,B]^{-1})_{[i][j]}\Vert \leq \Gamma_U \rho_U^{\lceil \frac{\Delta_U(i,j)-1}{2}\rceil}$, if $B\in\mathbb{B}_U$.
\end{lemma}
\iffull\begin{proof}
  By definition, $\underline{\sigma}_U^2 I \preceq \widetilde{\bH}^2\preceq  \overline{\sigma}_U^2I$, and thus
 \begin{align}\label{eqn:factor}
   \frac{\underline{\sigma}_U^2-\overline{\sigma}_U^2}{\underline{\sigma}_U^2+\overline{\sigma}_U^2} I\preceq I-\frac{2}{\underline{\sigma}_U^2+\overline{\sigma}_U^2}\widetilde{\bH}^2
   \preceq \frac{-\underline{\sigma}_U^2+\overline{\sigma}_U^2}{\underline{\sigma}_U^2+\overline{\sigma}_U^2} I .
  \end{align}
 Moreover, 
 \begin{subequations}\label{eqn:key-1}
   \begin{align}
     \widetilde{\bH}^{-1} &= \frac{2}{\underline{\sigma}_U^2+\overline{\sigma}_U^2} \widetilde{\bH}\left(\frac{2}{\underline{\sigma}_U^2+\overline{\sigma}_U^2}\widetilde{\bH}^2 \right)^{-1}\\
   %% &= \frac{2}{\underline{\sigma}_U^2+\overline{\sigma}_U^2} \left(I-(I-\frac{2}{\underline{\sigma}_U^2+\overline{\sigma}_U^2}\widetilde{\bH}^2 )\right)^{-1}\\
     %% &= \frac{2}{\underline{\sigma}_U^2+\overline{\sigma}_U^2}\widetilde{\bH} \sum_{q=0}^\infty \left(I-\frac{2}{\underline{\sigma}_U^2+\overline{\sigma}_U^2}\widetilde{\bH}^2 \right)^q,\\
     &=\frac{2}{\underline{\sigma}_U^2+\overline{\sigma}_U^2} \sum_{q=0}^\infty \widetilde{\bH}\left(I-\frac{2}{\underline{\sigma}_U^2+\overline{\sigma}_U^2}\widetilde{\bH}^2 \right)^q,\label{eqn:key-1-c}
 \end{align}
 \end{subequations}
 where the second equality follows from \cite[Theorem 5.6.9 and Corollay 5.6.16]{horn2012matrix}, \eqref{eqn:factor}, and the fact that the series in \eqref{eqn:key-1-c} converges.  By Lemma \ref{lem:claim}, if $\Delta_U(i,j)>2q+1$, then
 \begin{align}\label{eqn:sparsity}
   \left(\widetilde{\bH}\left(I-\frac{2}{\underline{\sigma}_U^2+\overline{\sigma}_U^2}\widetilde{\bH}^2 \right)^q\right)_{[i][j]}=0
  \end{align}
  holds. By extracting submatrices from \eqref{eqn:key-1}, one establishes
  \begin{align*}
    (\widetilde{\bH}^{-1})_{[i][j]}=\frac{2}{\underline{\sigma}_U^2+\overline{\sigma}_U^2} \sum_{q=q_0}^\infty \left(\widetilde{\bH}\left(I-\frac{2\widetilde{\bH}^2}{\underline{\sigma}_U^2+\overline{\sigma}_U^2} \right)^q\right)_{[i][j]} ,
  \end{align*}
  where $q_0:=\lceil \frac{\Delta_U(i,j)-1}{2}\rceil$ since the summation over $q=0,\cdots,q_0-1$ adds up to zero by \eqref{eqn:sparsity}. Using the triangle inequality and the fact that the matrix norm of a submatrix is always smaller than that of the original matrix, we have
  \begin{align}\label{eqn:Hbound}
    \Vert (\widetilde{\bH}^{-1})_{[i][j]} \Vert &\leq \frac{2}{\underline{\sigma}_U^2+\overline{\sigma}_U^2} \sum_{q=q_0}^\infty \left\Vert\widetilde{\bH}\left(I-\frac{2\widetilde{\bH}^2}{\underline{\sigma}_U^2+\overline{\sigma}_U^2}\right)^q\right\Vert\nonumber\\
    %% &\leq \frac{2}{\underline{\sigma}_U^2+\overline{\sigma}_U^2} \sum_{q=\lceil \frac{\Delta_U(i,j)}{2}\rceil}^\infty \left\Vert I-\frac{2\widetilde{\bH}^2}{\underline{\sigma}_U^2+\overline{\sigma}_U^2} \right\Vert^q\nonumber\\
    &\leq \frac{2}{\underline{\sigma}_U^2+\overline{\sigma}_U^2} \sum_{q=q_0}^\infty \overline{\sigma}_U\left(\frac{\overline{\sigma}_U^2-\underline{\sigma}_U^2}{\overline{\sigma}_U^2+\underline{\sigma}_U^2}\right)^q\nonumber\\
    &\leq \frac{\overline{\sigma}_U}{\underline{\sigma}_U^2}\left(\frac{\overline{\sigma}_U^2-\underline{\sigma}_U^2}{\overline{\sigma}_U^2+\underline{\sigma}_U^2}\right)^{ \lceil(\Delta_U(i,j)-1)/2\rceil} ,
  \end{align}
  where the second inequality follows from the submultiplicativity of the matrix norm and the fact that the induced 2-norm of a symmetric matrix is equal to its largest eigenvalue; the last inequality follows from geometric series.
\end{proof}\fi

\begin{proof}[Proof of Theorem \ref{thm:eds}]
  From Lemma \ref{lem:basic-2}, we have
  \begin{align}\label{eqn:eds-1}
    \bz^*_{U}( \bd)-\bz^*_{U}( \bd')= \sum_{k=1}^{N_d} \bz^{B_k}_U(\bd^{s_k}_U)-\bz^{B_k}_U(\bd^{s_{k-1}}_U).
  \end{align}
The solution does not {\it jump} when the basis changes, because of solution uniqueness. From the block multiplication formula and $\bd^{s_k}_U-\bd^{s_{k-1}}_U=(s_{k}-s_{k-1})(\bd_U-\bd'_U)$, we have
  \begin{align}\label{eqn:eds-2}
    &\left(\bz^B_U(\bd_U^{s_{k}})[B]-\bz^B_U(\bd_U^{s_{k-1}})[B]\right)_{[i]} \\
    &\quad= \sum_{j\in U}\left(\widetilde{\bH}^{-1}\right)_{[i][j]}(s_{k}-s_{k-1})(\bd_U[B]-\bd'_U[B])_{[j]}.\nonumber
  \end{align}
  From the definition of bases and basic solutions, we have:
  \begin{align}\label{eqn:eds-2.5}
    &\Vert\left(\bz^B_U(\bd_U^{s_{k}})[B]-\bz^B_U(\bd_U^{s_{k-1}})[B]\right)_{[i]}\Vert\\
    &\quad= \left\Vert T_{i\leftarrow U}(\bz^B_{U}( \bd^{s_k})-\bz^B_{U}( \bd^{s_{k-1}}))\right\Vert.\nonumber
  \end{align}
  From \eqref{eqn:eds-2}-\eqref{eqn:eds-2.5}, the triangle inequality, the submultiplicativity of matrix norms, Lemma \ref{lem:inv}, and the fact that the norm of a subvector is not greater than the norm of the original vector, 
  \begin{align}\label{eqn:eds-3}
    &\left\Vert T_{i\leftarrow U}(\bz^B_{U}( \bd^{s_k})-\bz^B_{U}( \bd^{s_{k-1}}))\right\Vert\\
    &\quad\leq \Gamma_U \rho_U^{\lceil(\Delta_U(i,j)-1)/2\rceil}
    \left(s_{k}-s_{k-1})\Vert d_j-d'_j\right\Vert.\nonumber
  \end{align}
  By multiplying \eqref{eqn:eds-1} by $T_{i\leftarrow U}$ and by applying \eqref{eqn:eds-3} and the triangle inequality, we obtain the result.
\end{proof}

The coefficient: 
\begin{equation*}
\Gamma_U \rho_U^{\lceil \frac{\Delta_U(i,j)-1}{2}\rceil}
\end{equation*}
in Theorem \ref{thm:eds} is the {\it componentwise Lipschitz constant} of the solution mapping $\bz^*_U(\cdot)$. Recall that $\rho_U\in(0,1)$; hence, the coefficient decays exponentially with the distance $\Delta_U(i,j)$. Therefore, the effect of the perturbation decays as one moves away from the perturbation location. 

\subsection{Convergence Analysis}
In this subsection, we formally establish the convergence of Algorithm \ref{alg:main}. To facilitate the later discussion, we first introduce some notation: $\bz^\dag_U:=\{z^\dag_i\}_{i\in U}$, where $z^\dag_i:=T_{i\leftarrow V}\bz^*_V(\bd_V)$ and $U\subseteq V$. Furthermore, we define $\Vert\cdot\Vert_{U,\infty}:=\max_{i\in U}\Vert T_{i\leftarrow U}(\cdot)\Vert$ and $\Vert\cdot\Vert_{U,U',1}:=\sum_{i\in U}\sum_{j\in U'}\Vert T_{i\leftarrow U}(\cdot)T^\top_{j\leftarrow U'}\Vert$.
%% Note that $\Vert \cdot \Vert_{U,1},\Vert \cdot \Vert_{U,\infty}:\mathbb{R}^{n_U}\rightarrow \mathbb{R}_{\geq 0}$ are vector norms; furthermore, we define .

\begin{assumption}\label{ass:sol-2}
  Assumption \ref{ass:sol} holds with $\mathbb{D}_{U}:=\{\bd_{U} - \bH_{-U} \bz_{-U}: \bz_{-U}\in\mathbb{R}^{n_{-U}}\}$ for any $U\subseteq V$.
\end{assumption}
Here $n_{-U}=\sum_{i\in N_V(U)}n_i$. While Assumption \ref{ass:sol-2} is  strong, we believe it can be relaxed; but doing so meaningfully is a technically extensive endeavor beyond the scope of this communication format. We note, however, that it is always satisfied for bound constraints and for augmented Lagrangian reformulations of the original problem by using slacks to obtain only bound inequality constraints and then penalizing all equality constraints, which can approximate the solution set of the original problem arbitrarily well.

\begin{assumption}\label{ass:omega}
  $\omega:=\min_{k\in\mathbb{I}_{1:K}}\Delta_V(V_k,N_V(W_k))-1\geq 1$.
\end{assumption}
Here, we abuse the notation by letting $\Delta_V(U,U'):=\min_{u\in U,u'\in U'}\Delta_V(u,u')$, where $U,U'\subseteq V$.
We call $\omega$ the {\it size of overlap}. Note that an overlapping partition $\{W_k\}_{k=1}^K$ with size $\omega$ can be constructed from a non-overlapping partition $\{V_k\}_{k=1}^K$ by expanding each original subdomain using $W_k=\{i\in V: \Delta_V(i,V_k)\leq \omega\}$.

\begin{assumption}\label{ass:limit}
  $\displaystyle\underline{\sigma}:=\inf_{U\subseteq \mathcal{V}}\underline{\sigma}_U >0$, $\displaystyle\overline{\sigma}:=\sup_{U\subseteq \mathcal{V}}\overline{\sigma}_U <\infty$.
  %% \item[(b)] There exists a polynomial $p:\mathbb{R}\rightarrow \mathbb{R}$ such that $R_V\leq p (D_V)$ holds, where $D_V$ is the diameter of $G(V,E)$.
  %% \end{itemize}
  %% Then $\alpha_V(\omega)\leq p(D_V)\Gamma \rho^{\lceil\omega/2\rceil-1}$, where
  %% $\Gamma:=\frac{\overline{\sigma}}{\underline{\sigma}^2}$,
  %% $\rho := \frac{\overline{\sigma}^2-\underline{\sigma}^2}{\overline{\sigma}^2+\underline{\sigma}^2}$.
\end{assumption}
Assumption \ref{ass:limit} trivially holds if the parent graph $\mathcal{G}(\mathcal{V},\mathcal{E})$ is finite, but it may be violated if the parent graph is infinite. In particular, $\underline{\sigma}_V$ or $\overline{\sigma}_V$ may tend to zero or infinity as $V$ grows. Checking the validity of Assumption \ref{ass:limit} for the infinite parent graph case is beyond the scope of this paper; sufficient conditions for this to hold will be studied in the future.

We note, however, that Assumptions \ref{ass:sol-2} and \ref{ass:limit} hold for augmented Lagrangian reformulations with bounded data when the objective matrix has bounded entries and is strongly diagonally dominant, which is a way we can approximate most QPs with some regularization; in contrast, Assumption \ref{ass:omega} can be satisfied by construction. Therefore our setup contains a large set of problems or close approximations. 

\begin{lemma}\label{lem:con}
  Let Assumption \ref{ass:sol-2} hold. For any $U\subseteq V$ we have that $\bz^\dag_U = \bz^*_{U}(\bd_{U} - \bH_{-U} \bz^\dag_{-U})$.
\end{lemma}
\begin{proof}
  Since $P_U(\cdot)$ is a convex QP, the KKT conditions are necessary and sufficient for optimality. By Assumption \ref{ass:sol-2}, the primal-dual solution is unique; therefore, it suffices to prove that $\bz^\dag_U$ satisfies the KKT conditions of $P_U(\bd_{U} - \bH_{-U}  \bz^\dag_{-U})$. From the KKT conditions of $P_V(\bd_V)$, we have
  \begin{align}\label{eqn:kkt-cen}
      &\bQ_V \bx^\dag_V + \bA_V^\top \blambda_V^\dag=\bff_V,\quad
      \bA_V^{\mathcal{E}} \bx_V^\dag=\bg_V^{\mathcal{E}},\quad
      \bA_V^{\mathcal{I}} \bx_V^\dag\geq\bg_V^{\mathcal{I}}\nonumber\\
      &\blambda_V^{\dag,\mathcal{I}} \geq 0,\quad
      \diag(\blambda_V^{\dag,\mathcal{I}}) (\bA_V^{\mathcal{I}} \bx_V^\dag-\bg_V^{\mathcal{I}}) = 0.
  \end{align}
  By extracting the rows associated with $U$ and rearranging equations and inequalities, we obtain
  \begin{align}\label{eqn:kkt-sub}
    &\bQ_{U} \bx^\dag_U + \bA_{U}^\top \blambda^\dag_U = \bff_{U} -\bQ_{-U}\bx^\dag_{-U}- \bA^\top_{-U}\blambda^\dag_{-U}\\
    &\bA^{\mathcal{E}}_{U}\bx^\dag_U =\bg^{\mathcal{E}}_{U} -\bA^{\mathcal{E}}_{-U} \bx^\dag_{-U}\quad
    \bA^{\mathcal{I}}_{U} \bx^\dag_U \geq\bg^{\mathcal{I}}_{U} -\bA^{\mathcal{I}}_{-U}  \bx^\dag_{-U},\nonumber\\
    &\blambda^{\dag,\mathcal{I}}_U \geq 0,\quad\diag(\blambda^{\dag,\mathcal{I}}_U)(\bA^{\mathcal{I}}_{U} \bx^\dag_U -\bg^{\mathcal{I}}_{U} +\bA_{-U}^{\mathcal{I}} \bx^\dag_{-U})= 0.\nonumber
  \end{align}
  %% Here $T_{\mathcal{I}\leftarrow\mathcal{J}}^x:=\{\delta_{ij}I_{r_i\times r_i}\}_{i\in \mathcal{I},j\in\mathcal{J}}$, and $T_{\mathcal{I}\leftarrow\mathcal{J}}^\lambda:=\{\delta_{ij}I_{m_i\times m_i}\}_{i\in \mathcal{I},j\in\mathcal{J}}$, and $\mathcal{I},\mathcal{J}\subseteq V$.
  Here, note that $\bA_{-U}:=\{A_{i,j}\}_{i\in U, j\in N_V(U)}$ and $\bA^\top_{-U}:=\{A_{j,i}^\top\}_{i\in U,j\in N_V(U)}$ (they are not transpose to each other).
  Conditions \eqref{eqn:kkt-sub} imply that $\bz^\dag_U$ satisfies the KKT conditions for $P_U(\bd_{U} - \bH_{-U} \bz^\dag_{-U})$. Thus, it is the solution.
\end{proof}

Lemma \ref{lem:con} provides a form of consistent subproblems whose solution recovers a piece of the full solution as long as the complementary solution information is accurate. Thus, it justifies the subproblem formulation in Algorithm \ref{alg:main}.

% \begin{remark}
%   Lemma \ref{lem:con} reveals that the algorithm applies an (overlapping) block-Jacobi scheme to the saddle point problem for the  Lagrangian. In particular, Algorithm \ref{alg:main} is equivalent to performing, for $\ell=0,1,\cdots$,
%   \begin{align*}
%     \bz^{(\ell+1)}_{V_k}=T_{V_k\leftarrow W_k}\aminmax_{\bz_{W_k}\in\prod_{i\in W_k}\mathbb{Z}_i} L(\bz_{W_k}; \bz^{(\ell)}_{-W_k}),\; k\in\mathbb{I}_{1:K} ,
%   \end{align*}
%   where $L(\cdot)$ is the Lagrangian of $P_V(\bd_V)$, $\mathbb{Z}_i:=\mathbb{R}^{r_i}\times \mathbb{R}^{m^\mathcal{E}_i}\times \mathbb{R}^{m^\mathcal{I}_i}_{\geq 0}$, and $\aminmax(\cdot)$ denotes the saddle point where the given function is minimized over the primal directions and maximized over the dual directions.
% \end{remark}

The main result of the paper is stated as follows.

\begin{theorem}[Convergence of Overlapping Schwarz]\label{thm:conv}
  Let Assumptions \ref{ass:sol-2}--\ref{ass:limit} hold. The sequence $\{\bz^{(\ell)}_V\}_{\ell=0}^\infty$ generated by Algorithm \ref{alg:main} satisfies
  \begin{align}\label{eqn:conv}
    \Vert \bz^{(\ell)}_V - \bz^\dag_V\Vert_{V,\infty} \leq \left(\alpha_V(\omega)\right)^{\ell} \Vert \bz^{(0)}_V - \bz^\dag_V\Vert_{V,\infty}.
  \end{align}
  Here, $\alpha_V(\omega):= R_V\Gamma\rho^{\lceil (\omega-1)/2\rceil} $;
  $\Gamma:={\overline{\sigma}}/{\underline{\sigma}^2}$; 
  $\rho := (\overline{\sigma}^2-\underline{\sigma}^2)/(\overline{\sigma}^2+\underline{\sigma}^2)$; and
  $R_V:= \displaystyle\max_{U\subseteq V} \Vert\bH_{-U}\Vert_{U,-U,1}$.
  %% $\displaystyle\underline{\sigma}:=\inf_{U\subseteq \mathcal{V}}\underline{\sigma}_k$;
  %% $\displaystyle\overline{\sigma}:=\sup_{U\subseteq \mathcal{V}}\overline{\sigma}_k$.
\end{theorem}

\begin{proof}
  By Lemma \ref{lem:con}, we have that, for any $i \in V_k$,
  \begin{align*}
    \left\Vert z^{(\ell+1)}_i - z^\dag_i\right\Vert
    &=\Big\Vert T_{i\leftarrow W_k}\bz^*_{W_k}(\bd_{W_k}-\bH_{-W_k} \bz^{(\ell)}_{-W_k})\\
    &\quad\quad-T_{i\leftarrow W_k}\bz^*_{W_k}(\bd_{W_k}-\bH_{-W_k} \bz^{\dag}_{-W_k}) \Big\Vert \\
    &\leq \sum_{j\in W_k} \Gamma_{W_k}\rho^{\lceil (\Delta_{W_k}(i,j)-1)/2\rceil}_{W_k}\\
    &\quad\quad\times\left\Vert T_{j\leftarrow W_k}\bH_{-W_k} (\bz^{(\ell)}_{-W_k}-\bz^\dag_{-W_k})\right\Vert.
  \end{align*}
  Here, the inequality follows from Theorem \ref{thm:eds}.  A key observation is that $T_{j\leftarrow W_k}\bH_{-W_k}\neq 0$ only if $\Delta_V(j,N_V(W_k))=1$. Such a $j$ satisfies $\Delta_{W_k}(i,j)\geq\omega$, for any $i \in V_k$, by the definition of $\omega$, the triangle inequality for $\Delta_V(\cdot,\cdot)$, and $\Delta_{W_k}(i,j)\geq \Delta_{V}(i,j)$. Therefore, for any $i\in V$, we have:
\begin{align*}
  &\left\Vert z^{(\ell+1)}_i - z^\dag_i\right\Vert\\
  &\leq\max_{k\in\mathbb{I}_{1:K}}\underbrace{\Vert \bH_{-W_k}\Vert_{W_k,-W_k,1} \Gamma_{W_k}\rho^{\lceil \frac{\omega-1}{2}\rceil}_{W_k}}_{\alpha_k}
  \left\Vert\bz^{(\ell)}_V-\bz^\dag_V \right\Vert_{V,\infty}
\end{align*}
One can show that $\max_{k\in\mathbb{I}_{1:K}}\alpha_k\leq \alpha_V(\omega)$ with Assumptions \ref{ass:omega}--\ref{ass:limit}. This implies $\Vert \bz^{(\ell+1)}_V- \bz^\dag_V \Vert_{V,\infty} \leq \alpha_V(\omega) \Vert \bz^{(\ell)}_V- \bz^*_V \Vert_{V,\infty}$, which yields \eqref{eqn:conv}.
\end{proof}

The upper bound of convergence rate $\alpha_V(\omega)$ decays with an increase in the size of overlap $\omega$ (recall that $\rho\in(0,1)$). In the case of a finite parent graph, Algorithm \ref{alg:main} converges if $\omega$ is sufficiently large, since either (i) $\alpha_V(\omega)<1$ or (ii) each subproblem becomes the full problem (the solution is obtained in one iteration). 

If the parent graph is infinite (such a setting is relevant for time grids, unbounded physical space domains, and scenario trees), we can make $V$ arbitrarily large, and thus the limiting behavior of $\alpha_V$ is of interest. Theorem \ref{thm:conv} suggests that constant $\Gamma$ and the exponential decay factor $\rho$ do not deteriorate as $V$ grows, but $R_V$ may grow with $V$. Here, $R_V$ represents the maximum coupling between the subdomains $U\subseteq V$ with its surroundings $V\setminus U$. Accordingly, the growth of $R_V$ is determined by the topology (as long as $\Vert H_{i,j}\Vert$ are uniformly bounded). In particular, if the parent graph $\mathcal{G}(\mathcal{V},\mathcal{E})$ is finite-dimensional (e.g., grid points in $\mathbb{I}^{d}$ with $d<\infty$), $R_V$ grows in a polynomial manner with the diameter of $V$. In such a case, a sufficiently large $\omega$ can be found so that the decay of $\rho^{\lceil (\omega-1)/2\rceil}$ offsets the growth of $R_V$ (an exponentially decaying function times a polynomial converges). On the other hand, in the case of scenario trees, where $R_V$ may grow exponentially, the upper bound derived in Theorem \ref{thm:conv} may not provide a tight upper bound, and $\alpha_V(\omega)$ may diverge as $V$ grows no matter how large $\omega$ is.

Note that there exists an inherent {\em trade-off} between the convergence rate and subproblem complexity. The convergence rate improves with $\omega$ but the subproblem solution times also increase with $\omega$. Therefore, to achieve maximum performance, one needs to tune $\omega$.
\vspace{-0.1in}
\subsection{Monitoring Convergence}
Convergence can be monitored by checking the residuals to the KKT conditions of the full problem $P_V(\bd_V)$. However, a more convenient surrogate of the full KKT residuals can be derived as follows.
\begin{proposition} Suppose that Assumptions \ref{ass:sol-2}--\ref{ass:omega} hold, and let $\{\bz_V^{(\ell)}\}_{\ell=0}^\infty$ be a sequence generated by Algorithm \ref{alg:main} and $\bz^{(\ell,k)}_{W_k}:=\bz^*_{W_k}(\bd_{W_k}-\bH_{-W_k}\bz^{(\ell-1)}_{-W_k})$. Then $\bz_V^{(\ell)}\rightarrow \bz^\dag_V$ as $\ell\rightarrow \infty$ if the following holds for $k\in\mathbb{I}_{1:K}$:
  \begin{align}\label{eqn:residual}
  E^{(\ell)}_k :=  \bz^{(\ell,k)}_{N_V(V_k)} - \bz^{(\ell)}_{N_V(V_k)}\rightarrow 0,\;\text{as}\;\ell\rightarrow\infty.
\end{align}
\end{proposition}
\iffull\begin{proof}
  We make two observations. (i) The KKT conditions of $P_{W_k}(\bd_{W_k}-\bH_{-W_k}\bz^{(\ell-1)}_{-W_k})$ and the fact that $N_V[V_k]\subseteq W_k$ (from Assumption \ref{ass:omega}) imply that \eqref{eqn:kkt-sub} holds when we replace $U\leftarrow V_k$, $\bz^\dag_{U} \leftarrow \bz^{(\ell,k)}_{V_k}$, and $\bz^\dag_{-U} \leftarrow \bz^{(\ell,k)}_{N_V(V_k)}$, for any $k\in\mathbb{I}_{1:K}$. (ii) The residual of the KKT systems can be obtained from the residuals to \eqref{eqn:kkt-sub} by replacing $U\leftarrow V_k$, $\bz^\dag_{U} \leftarrow \bz^{(\ell,k)}_{V_k}$, and $\bz^\dag_{-U} \leftarrow \bz^{(\ell)}_{N_V(V_k)}$ and collecting the residuals for $k\in\mathbb{I}_{1:K}$. Note that the residuals of the KKT conditions \eqref{eqn:kkt-sub} are continuous with respect to $\bz^\dag_{-U}$. By using a continuity argument, \eqref{eqn:residual} and observations (i)--(ii), the limit points of $\{\bz^{(\ell)}\}_{\ell=1}^\infty$ satisfy \eqref{eqn:kkt-sub} for $k\in\mathbb{I}_{1:K}$ (i.e., \eqref{eqn:kkt-cen} holds). By the uniqueness of the solution such a limit point is unique, and this implies $\bz_V^{(\ell)}\rightarrow \bz^\dag_V$ as $\ell\rightarrow\infty$.
\end{proof}\fi
We can define primal-dual errors as $\displaystyle\epsilon_{\text{pr}}:=\max_{k\in\mathbb{I}_{1:K}}\Vert \text{pr}(E^{(\ell)}_k)\Vert_\infty$ and $\displaystyle\epsilon_{\text{du}}:=\max_{k\in\mathbb{I}_{1:K}}\Vert \text{du}(E^{(\ell)}_k)\Vert_\infty$, where pr$(\cdot)$ and du$(\cdot)$ extract the indices associated with primal variables and dual variables, respectively. Convergence criteria can thus be set to the following: 
Stop if $(\epsilon_{\text{pr}}<\epsilon_{\text{pr}}^{\text{tol}})\land(\epsilon_{\text{du}}<\epsilon_{\text{du}}^{\text{tol}})$.

\section{Numerical Example}\label{sec:num}
Consider the regularized DC OPF problem \cite{sun2010dc} over a network $\mathcal{G}(\mathcal{V},\mathcal{E})$:
\vspace{-0.1in}
\begin{subequations}\label{eqn:dcopf}
  \begin{align}
    &\min_{\substack{\{\theta_i\}_{i\in \mathcal{V}}\\\{P_q\}_{q\in \Omega}}}\;\sum_{q\in \Omega} c_{q,1}P_{q} +c_{q,2}P^2_q + \frac{\gamma}{2}\sum_{i\in \mathcal{V}} (\theta_i-\theta_j)^2\\
    &\st\; \sum_{q\in \Omega_i} P_q  - \sum_{j\in N_\mathcal{V}[i]} B_{i,j} (\theta_i-\theta_j) = P^L_i,\; i\in \mathcal{V}\\
    &\qquad\underline{P}_q\leq P_{q} \leq \overline{P}_q,\;q\in \Omega,\; \theta_{i} = \theta^{\text{ref}}_i,\; i \in \mathcal{V}^{\text{ref}}\\
    &\qquad-\overline{\theta}_{i,j}\leq \theta_i -\theta_j\leq\overline{\theta}_{i,j},\; \{i,j\}\in \mathcal{E}.\label{eqn:dcopf-angle-limit}
  \end{align}
\end{subequations}
Here, $\Omega_i$ is the set of generators in node $i$; $\Omega:=\bigcup_{i\in \mathcal{V}}\Omega_i$ is the set of all generators; $\mathcal{V}^{\text{ref}}$ is the set of reference nodes; $\theta_i$ are the voltage angles; $P_q$ are the active power generations; $c_{q,1}$ and $c_{q,2}$ are the generation cost coefficients; $\gamma$ is the regularization coefficient; $B_{i,j}$ are the line susceptances; $P^L_i$ are the active power loads; $\underline{P}_q$ and $\overline{P}_q$ are the lower and upper  bounds of active power generations, respectively; $\overline{\theta}_{i,j}$ are the voltage angle separation limits; and $\theta^{\text{ref}}_i$ are reference voltage angles. The problems are modeled in the algebraic modeling language {\tt JuMP} \cite{dunning2017jump} and solved with the nonlinear programming solver {\tt Ipopt} \cite{wachter2006implementation}. The $9,241$-bus test case obtained from pglib-opf (v19.05) is used \cite{babaeinejadsarookolaee2019power}. We modified the data by placing artificial storage with infinite capacity and high charge/discharge cost in each node. The network node set is partitioned into $16$ subdomains using the graph partitioning tool {\tt Metis} \cite{karypis1998fast}, and each subdomain is expanded to obtain an overlapping partition. The Schwarz scheme is run on a multicore parallel computing server (shared memory and 32 CPUs of Intel Xeon CPU E5-2698 v3 running at 2.30 GHz) using the {\tt Julia} package {\tt Distributed.jl}. One master process and 16 worker processes are used (one process per one subproblem). We use $\gamma=10^5$, $\epsilon^{\text{tol}}_{\text{pr}}=10^{-2}$, and $\epsilon^{\text{tol}}_{\text{du}}=10^2$. The scripts to reproduce the results can be found here \url{https://github.com/zavalab/JuliaBox/tree/master/SchwarzQPcons}.

Convergence results are shown in Fig. \ref{fig:results}; we vary the size of overlap $\omega$ and show the evolution of the objective value and primal-dual error. The black dashed line represents the optimal objective value (obtained by solving the problem with {\tt Ipopt}) and the error tolerances. The total computation times and iteration counts are compared in Table \ref{tbl:results}. We can see that increasing $\omega$ accelerates convergence (reduces the iteration count roughly proportionally with the size of the overlap) but does not always reduce the computation time. The reason is that the  subproblem complexity increases with $\omega$; thus, one can see that an optimal overlap exists. {However, the overlapping scheme is still considerably slower than the centralized solver; in particular, {\tt Ipopt} solves the full problem in 1.92 seconds. We expect that the computational savings can achieved if the problem is bigger. In the future, we will investigate the algorithm with several large-scale problems, and study the conditions under which the overlapping scheme becomes more effective.}

\begin{figure}
  \centering
  \includegraphics[width=.45\textwidth,trim={0 .8cm 0 0},clip]{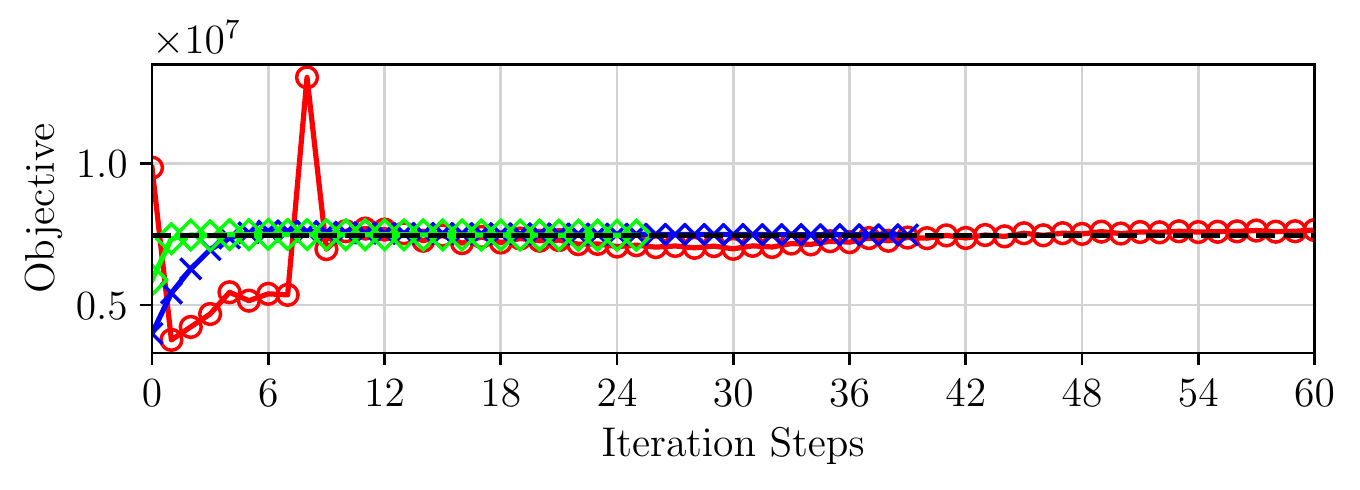}\\
  \includegraphics[width=.45\textwidth,trim={0 .8cm 0 0},clip]{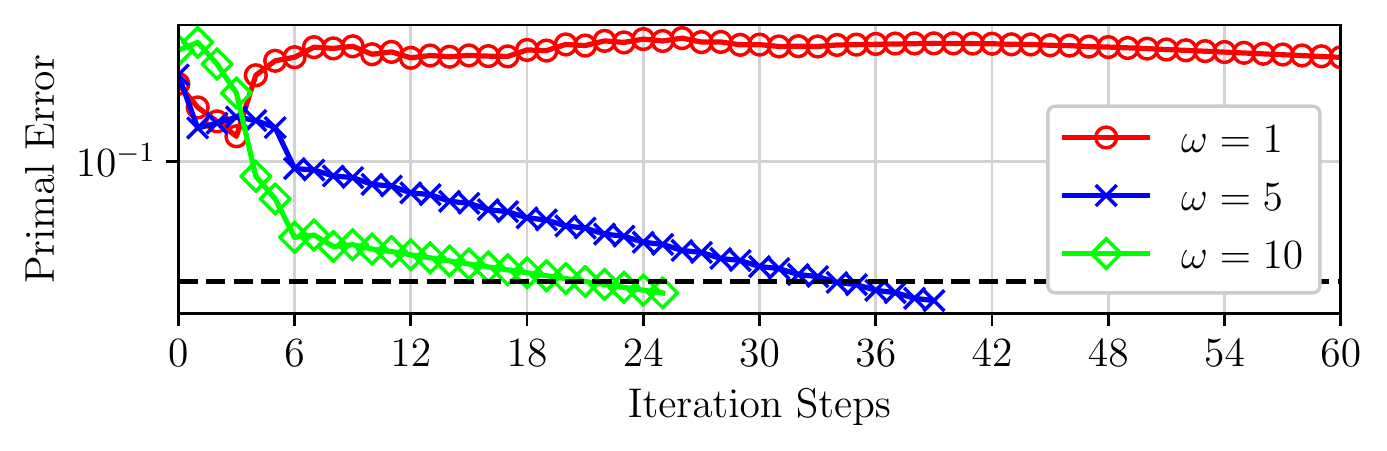}\\%\\
  \includegraphics[width=.45\textwidth,trim={0 0 0 0},clip]{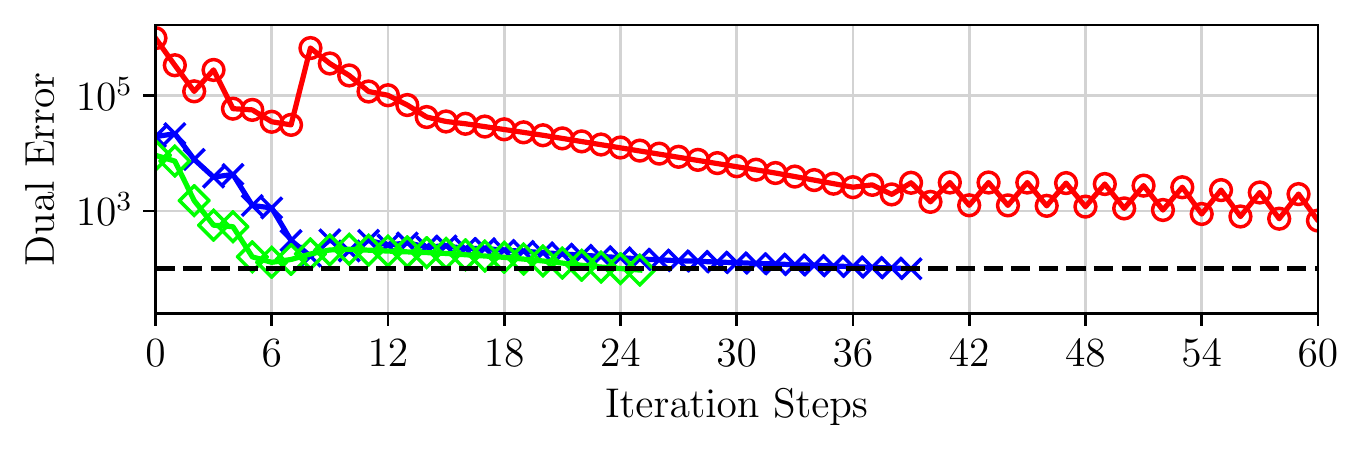}%\\
  \caption{Profiles for objective and primal error for different overlap sizes.\label{fig:results}}
\end{figure}

\begin{table}
  \caption{Effect of overlap on solution times and iteration counts}\label{tbl:results}
  \centering
  \begin{tabular}{|c|c|c|c|c|}
    \hline
    &$\omega=1$&$\omega=5$&$\omega=10$\\
    \hline
    Solution time&
    42.5 sec&
    17.7 sec&
    24.2 sec\\
    \hline
    Iterations&
     $190$ iter&
     $39$ iter&
     $25$ iter\\
     \hline
  \end{tabular}
\end{table}

\section{Conclusions}
We have presented an overlapping Schwarz algorithm for solving constrained quadratic programs. We show that convergence relies on an exponential decay of the sensitivity result, which we establish for the setting of interest. The behavior of the algorithm was demonstrated by using a large DC optimal power flow problem.  In future work we will study convergence in a nonlinear setting, and we will conduct more extensive benchmark studies.

\section*{Acknowledgments}
This material is based upon work supported by the U.S. Department of Energy, Office of Science, Office of Advanced Scientific Computing Research (ASCR) under Contract DE-AC02-06CH11347 and by NSF through award CNS-1545046. We also acknowledge partial support from the National Science Foundation under award NSF-EECS-1609183.

\bibliographystyle{IEEEtran}
\bibliography{Schwarz_QP}

\begin{thebibliography}{10}
\providecommand{\url}[1]{#1}
\csname url@samestyle\endcsname
\providecommand{\newblock}{\relax}
\providecommand{\bibinfo}[2]{#2}
\providecommand{\BIBentrySTDinterwordspacing}{\spaceskip=0pt\relax}
\providecommand{\BIBentryALTinterwordstretchfactor}{4}
\providecommand{\BIBentryALTinterwordspacing}{\spaceskip=\fontdimen2\font plus
\BIBentryALTinterwordstretchfactor\fontdimen3\font minus
  \fontdimen4\font\relax}
\providecommand{\BIBforeignlanguage}[2]{{%
\expandafter\ifx\csname l@#1\endcsname\relax
\typeout{** WARNING: IEEEtran.bst: No hyphenation pattern has been}%
\typeout{** loaded for the language `#1'. Using the pattern for}%
\typeout{** the default language instead.}%
\else
\language=\csname l@#1\endcsname
\fi
#2}}
\providecommand{\BIBdecl}{\relax}
\BIBdecl

\bibitem{rawlings2009model}
J.~B. Rawlings and D.~Q. Mayne, ``Model predictive control: Theory and
  design,'' 2009.

\bibitem{yu2019advanced}
Z.~J. Yu and L.~T. Biegler, ``Advanced-step multistage nonlinear model
  predictive control: Robustness and stability,'' \emph{Journal of Process
  Control}, vol.~84, pp. 192--206, 2019.

\bibitem{pereira1991multi}
M.~V. Pereira and L.~M. Pinto, ``Multi-stage stochastic optimization applied to
  energy planning,'' \emph{Mathematical programming}, vol.~52, no. 1-3, pp.
  359--375, 1991.

\bibitem{biegler2007real}
L.~T. Biegler, O.~Ghattas, M.~Heinkenschloss, D.~Keyes, and B.~van
  Bloemen~Waanders, \emph{Real-time PDE-constrained Optimization}.\hskip 1em
  plus 0.5em minus 0.4em\relax SIAM, 2007.

\bibitem{lavaei2011zero}
J.~Lavaei and S.~H. Low, ``Zero duality gap in optimal power flow problem,''
  \emph{IEEE Transactions on Power Systems}, vol.~27, no.~1, pp. 92--107, 2011.

\bibitem{coffrin2018powermodels}
C.~Coffrin, R.~Bent, K.~Sundar, Y.~Ng, and M.~Lubin, ``Powermodels. jl: An
  open-source framework for exploring power flow formulations,'' in \emph{2018
  Power Systems Computation Conference (PSCC)}.\hskip 1em plus 0.5em minus
  0.4em\relax IEEE, 2018, pp. 1--8.

\bibitem{zlotnik2015optimal}
A.~Zlotnik, M.~Chertkov, and S.~Backhaus, ``Optimal control of transient flow
  in natural gas networks,'' in \emph{2015 54th IEEE conference on decision and
  control (CDC)}.\hskip 1em plus 0.5em minus 0.4em\relax IEEE, 2015, pp.
  4563--4570.

\bibitem{rantzer2009dynamic}
A.~Rantzer, ``Dynamic dual decomposition for distributed control,'' in
  \emph{2009 American Control Conference}.\hskip 1em plus 0.5em minus
  0.4em\relax IEEE, 2009, pp. 884--888.

\bibitem{conte2012computational}
C.~Conte, T.~Summers, M.~N. Zeilinger, M.~Morari, and C.~N. Jones,
  ``Computational aspects of distributed optimization in model predictive
  control,'' in \emph{2012 IEEE 51st IEEE Conference on Decision and Control
  (CDC)}.\hskip 1em plus 0.5em minus 0.4em\relax IEEE, 2012, pp. 6819--6824.

\bibitem{shin2019hierarchical}
S.~Shin, P.~Hart, T.~Jahns, and V.~M. Zavala, ``A hierarchical optimization
  architecture for large-scale power networks,'' \emph{IEEE Transactions on
  Control of Network Systems}, 2019.

\bibitem{engelmann2020decomposition}
A.~Engelmann, Y.~Jiang, B.~Houska, and T.~Faulwasser, ``Decomposition of
  non-convex optimization via bi-level distributed aladin,'' \emph{IEEE
  Transactions on Control of Network Systems}, 2020.

\bibitem{jackson2003temporal}
J.~R. Jackson and I.~E. Grossmann, ``Temporal decomposition scheme for
  nonlinear multisite production planning and distribution models,''
  \emph{Industrial \& engineering chemistry research}, vol.~42, no.~13, pp.
  3045--3055, 2003.

\bibitem{lemarechal2001lagrangian}
C.~Lemar{\'e}chal, ``Lagrangian relaxation,'' in \emph{Computational
  combinatorial optimization}.\hskip 1em plus 0.5em minus 0.4em\relax Springer,
  2001, pp. 112--156.

\bibitem{boyd2011distributed}
S.~Boyd, N.~Parikh, E.~Chu, B.~Peleato, J.~Eckstein \emph{et~al.},
  ``Distributed optimization and statistical learning via the alternating
  direction method of multipliers,'' \emph{Foundations and
  Trends{\textregistered} in Machine learning}, vol.~3, no.~1, pp. 1--122,
  2011.

\bibitem{shin2018multi}
S.~Shin and V.~M. Zavala, ``Multi-grid schemes for multi-scale coordination of
  energy systems,'' in \emph{Energy Markets and Responsive Grids}.\hskip 1em
  plus 0.5em minus 0.4em\relax Springer, 2018, pp. 195--222.

\bibitem{carli2017distributed}
R.~Carli and M.~Dotoli, ``A distributed control algorithm for waterfilling of
  networked control systems via consensus,'' \emph{IEEE control systems
  letters}, vol.~1, no.~2, pp. 334--339, 2017.

\bibitem{kozma2015benchmarking}
A.~Kozma, C.~Conte, and M.~Diehl, ``Benchmarking large-scale distributed convex
  quadratic programming algorithms,'' \emph{Optimization Methods and Software},
  vol.~30, no.~1, pp. 191--214, 2015.

\bibitem{shin2020decentralized}
S.~Shin, V.~M. Zavala, and M.~Anitescu, ``Decentralized schemes with overlap
  for solving graph-structured optimization problems,'' \emph{IEEE Transactions
  on Control of Network Systems}, 2020.

\bibitem{frommer2001algebraic}
A.~Frommer and D.~B. Szyld, ``An algebraic convergence theory for restricted
  additive schwarz methods using weighted max norms,'' \emph{SIAM journal on
  numerical analysis}, vol.~39, no.~2, pp. 463--479, 2001.

\bibitem{na2019exponential}
S.~Na and M.~Anitescu, ``Exponential decay in the sensitivity analysis of
  nonlinear dynamic programming,'' \emph{arXiv preprint arXiv:1912.06734},
  2019.

\bibitem{shin2020diffusing}
S.~Shin and V.~M. Zavala, ``Diffusing-horizon model predictive control,''
  \emph{arXiv preprint arXiv:2002.08556}, 2020.

\bibitem{grüne2019sensitivity}
L.~Grüne, M.~Schaller, and A.~Schiela, ``Sensitivity analysis of optimal
  control for a class of parabolic pdes motivated by model predictive
  control,'' \emph{SIAM Journal on Control and Optimization}, vol.~57, no.~4,
  pp. 2753--2774, 2019.

\bibitem{grune2020exponential}
L.~Gr{\"u}ne, M.~Schaller, and A.~Schiela, ``Exponential sensitivity and
  turnpike analysis for linear quadratic optimal control of general evolution
  equations,'' \emph{Journal of Differential Equations}, vol. 268, no.~12, pp.
  7311--7341, 2020.

\bibitem{nocedal2006numerical}
J.~Nocedal and S.~Wright, \emph{Numerical optimization}.\hskip 1em plus 0.5em
  minus 0.4em\relax Springer Science \& Business Media, 2006.

\bibitem{hadigheh2007sensitivity}
A.~G. Hadigheh, O.~Romanko, and T.~Terlaky, ``Sensitivity analysis in convex
  quadratic optimization: simultaneous perturbation of the objective and
  right-hand-side vectors,'' \emph{Algorithmic Operations Research}, vol.~2,
  no.~2, pp. 94--94, 2007.

\bibitem{bonnans2013perturbation}
J.~F. Bonnans and A.~Shapiro, \emph{Perturbation analysis of optimization
  problems}.\hskip 1em plus 0.5em minus 0.4em\relax Springer Science \&
  Business Media, 2013.

\bibitem{horn2012matrix}
R.~A. Horn and C.~R. Johnson, \emph{Matrix analysis}.\hskip 1em plus 0.5em
  minus 0.4em\relax Cambridge university press, 2012.

\bibitem{sun2010dc}
J.~Sun and L.~Tesfatsion, ``{D}{C} optimal power flow formulation and solution
  using quadprogj,'' 2010.

\bibitem{dunning2017jump}
I.~Dunning, J.~Huchette, and M.~Lubin, ``Jump: A modeling language for
  mathematical optimization,'' \emph{SIAM Review}, vol.~59, no.~2, pp.
  295--320, 2017.

\bibitem{wachter2006implementation}
A.~W{\"a}chter and L.~T. Biegler, ``On the implementation of an interior-point
  filter line-search algorithm for large-scale nonlinear programming,''
  \emph{Mathematical programming}, vol. 106, no.~1, pp. 25--57, 2006.

\bibitem{babaeinejadsarookolaee2019power}
S.~Babaeinejadsarookolaee, A.~Birchfield, R.~D. Christie, C.~Coffrin,
  C.~DeMarco, R.~Diao, M.~Ferris, S.~Fliscounakis, S.~Greene, R.~Huang
  \emph{et~al.}, ``The power grid library for benchmarking {A}{C} optimal power
  flow algorithms,'' \emph{arXiv preprint arXiv:1908.02788}, 2019.

\bibitem{karypis1998fast}
G.~Karypis and V.~Kumar, ``A fast and high quality multilevel scheme for
  partitioning irregular graphs,'' \emph{SIAM Journal on scientific Computing},
  vol.~20, no.~1, pp. 359--392, 1998.

\end{thebibliography}

%\vspace{0.1cm}
%\begin{flushright}
%	\scriptsize \framebox{\parbox{2.5in}{Government License: The
%			submitted manuscript has been created by UChicago Argonne,
%			LLC, Operator of Argonne National Laboratory (``Argonne").
%			Argonne, a U.S. Department of Energy Office of Science
%			laboratory, is operated under Contract
%			No. DE-AC02-06CH11357.  The U.S. Government retains for
%			itself, and others acting on its behalf, a paid-up
%			nonexclusive, irrevocable worldwide license in said
%			article to reproduce, prepare derivative works, distribute
%			copies to the public, and perform publicly and display
%			publicly, by or on behalf of the Government. The Department of Energy will provide public access to these results of federally sponsored research in accordance with the DOE Public Access Plan. http://energy.gov/downloads/doe-public-access-plan. }}
%	\normalsize
%\end{flushright}	

\end{document}